\documentclass[reqno]{amsart}    

\usepackage[margin=1in]{geometry} 
\usepackage[parfill]{parskip}    
\usepackage{graphicx}            
\usepackage{epstopdf}            
\usepackage{yfonts}

\usepackage{amssymb}             
\usepackage{amsmath}
\usepackage{amsthm}              
\newcommand{\old}[1]{{}}

\usepackage{qtree}               
\usepackage{hyperref}            
\usepackage{algorithm}           
\usepackage{algorithmic}

\usepackage{ mathrsfs }
\usepackage{ marvosym }
\usepackage{ upgreek }
\usepackage{ dsfont }

\usepackage{xcolor}
\usepackage{cancel}

\usepackage{tikz}

\begin{document}


\newtheorem{thm}{Theorem}
\newtheorem{prop}[thm]{Proposition}
\newtheorem{conj}[thm]{Conjecture}
\newtheorem{lem}[thm]{Lemma}
\newtheorem{cor}[thm]{Corollary}
\newtheorem{ques}[thm]{Question}

\theoremstyle{definition}
\newtheorem{definition}[thm]{Definition}
\newtheorem{example}[thm]{Example}
\newtheorem{remark}[thm]{Remark}

\def\dd{{\operatorname{d}}}
\def\ds{\mathfrak{d}}
\def\vol{{\operatorname{vol}}}

\newtheorem*{thm*}{\normalfont\scshape Main Theorem}
\newtheorem*{conj*}{\normalfont\scshape Conjecture 1}

\begin{abstract}
Given a manifold $\mathcal{M} \subset \mathbb{R}^n$, we consider all codimension-1 submanifolds of $\mathcal{M}$ that satisfy the generalized Stokes' theorem and show that $\partial\mathcal{M}$ uniquely maximizes the associated entropy functional. This provides an information theoretic characterization of the duality expressed by Stokes' theorem, whereby a manifold's boundary is its `least informative' subset satisfying the Stokes relation.
\end{abstract}

\title{Stokes' theorem as an entropy-extremizing duality}
\today
\author{Daniel Lazarev}
\address{Department of Mathematics, Massachusetts Institute of Technology, Cambridge, MA, USA}
\email{dlazarev@mit.edu}  

\maketitle

Entropy has been used to recast and strengthen many foundational results in geometry and analysis, including the isoperimetric inequality, the Brunn-Minkowski inequality, and Young's inequality \cite{demcovtho91, CovTho91}. Beyond reproving known results, an entropic perspective can uncover new structure. For example, the celebrated entropy monotonicity result of Artstein, Ball, Barthe and Naor \cite{ArtBall04} shows that entropy not only characterizes the Gaussian extremizer in the central limit theorem, but also increases monotonically along the  sequence of normalized sums of i.i.d. random variables, thus suggesting a deeper gradient flow structure toward Gaussianity \cite{JorKinOtt98,Vil09}.

In this note, we show that Stokes’ theorem admits a natural reinterpretation as an entropy-extremizing duality. Given a differential form $\omega$ on a manifold $\mathcal{M}$, we identify $\partial \mathcal{M}$ as the maximum entropy representative  among all codimension-1 submanifolds of $\mathcal{M}$ that encode the same integral information about $\dd \omega$. In this formulation, entropy maximization serves as a variational principle that selects the canonical boundary dual to an exact form. This perspective provides a conceptual bridge between geometric-analytic dualities and information-theoretic optimality, suggesting a general framework in which classical dualities---such as those arising from convex analysis, Hodge theory, or functional inequalities---may be understood through entropic extremality \cite{ArtMilSza04, ArtMil09}. This entropy-based approach may also be extended beyond Stokes' theorem, hinting at a more general entropic calculus over cohomological or variational structures.

The generalized Stokes' theorem is given by,
\begin{thm}[Stokes' Theorem]
\label{thm:stokes}
    Let $\mathcal{M}$ be an orientable manifold of dimension $n$ with boundary $\partial \mathcal{M}$, and let $\omega$ be an $(n-1)$-form, with exterior derivative $\dd \omega$. Then, 
    \begin{equation}
\label{eq:stokes}
        \int_{\mathcal{M}} \dd \omega = \int_{\partial \mathcal{M}} \omega \;.
    \end{equation}
\end{thm}
The proof is constructive (in the sense of validating that $\partial \mathcal{M}$ satisfies \eqref{eq:stokes}) and follows, at its core, from the fundamental theorem of calculus and the definition of the derivative (see, for example, \cite{GuiPol74}). 
In general, however, there could be other subsets of $\mathcal{M}$ that satisfy (\ref{eq:stokes}).
We define the set of all submanifolds of codimension $1$ of $\mathcal{M}$ that satisfy the Stokes relation for a given $\omega$:
\begin{equation}
\label{eq:B}
    \mathcal{B}_{\omega, \mathcal{M}} := \left\{B \subset \mathcal{M}:  \int_{B} \omega  = \int_{\mathcal{M}} \dd \omega   \right\} .
\end{equation}
Every $B\in \mathcal{B}_{\omega, \mathcal{M}}$ thus serves as an ``alternative boundary" of $\mathcal{M}$, an $(n-1)$-dimensional subset of $\mathcal{M}$ over which $\omega$ is integrated, and whose integral must equal that of $\dd \omega$ over $\mathcal{M}$, per Stokes' theorem. 
Our main result, given in Theorem \ref{thm:entropy} below, shows that among all such $B\in \mathcal{B}_{\omega, \mathcal{M}}$ satisfying the Stokes relation, $\partial \mathcal{M}$ uniquely maximizes entropy. 
From an information theoretic perspective, this means $\partial \mathcal{M}$ serves as the only submanifold of $\mathcal{M}$ satisfying the constraint given by the Stokes relation but assuming no more information, or satisfying no additional constraints, other than that one.

We first define the measure corresponding to a form and its density.
\begin{definition}
    Given a Riemannian manifold and form, $(\mathcal{A}, \alpha)$, both of dimension $n$, we define the {\it measure} of $\alpha$ on $\mathcal{A}$ by
\begin{equation}
    \label{eq:mu}
    \mu_{\alpha, \mathcal{A}}(A) := \int_{\mathcal{A}} |\alpha|
,
\end{equation}
and we define the corresponding {\it density,} $\rho_\alpha$, as the Radon-Nikodym derivative with respect to Lebesgue measure,
\begin{equation}
    \label{eq:rho}
    \rho_{\alpha, \mathcal{A}} := \dfrac{d \mu_{\alpha,\mathcal{A}}}{d\vol} \,.
\end{equation}
\end{definition}
If the manifold over which the density is defined is clear, we will also write $\rho_{\alpha} :=\rho_{\alpha, \mathcal{A}}$ and $\mu_{\alpha} :=\mu_{\alpha, \mathcal{A}}$.

\begin{definition}
\label{def:S*-entropy}
The {\it $S^*$-entropy} of $(\mathcal{A}, \alpha)$ is defined by,
\begin{equation}
    \label{eq:S-entropy}
    S^*(\alpha, \mathcal{A}) := -\int_{\mathcal{A}} \rho_{\alpha} \log \left( \dfrac{\rho_{\alpha} }{C}\right) \; d\vol \;, 
\end{equation} 
where $C:= C_{\alpha,\mathcal{A}} = \max (1,  \lVert \rho_\alpha \rVert_\infty)$,
and $\lVert \rho_\alpha  \rVert_\infty$ is the supremum norm of $\rho$ over $\mathcal{A}$, $\lVert \rho_\alpha  \rVert_\infty := {\rm sup}\{ \rho_\alpha(x): x \in \mathcal{A}\}$. We take $0 \log0=0$.
\end{definition}

The `sup-normalization' in this definition allows for important properties of entropy seen in the discrete (Shannon) definition to be retained by ensuring that the argument of the logarithm does not exceed $1$ \cite{Laz23}, and reduces to the Shannon entropy if $\rho_\alpha$ is a discrete probability measure.

In defining the $S^*$-entropy we  do not normalize the density, which ensures that every point in the space is given the same weight when comparing the $S^*$-entropies of different subsets of the same ambient space, as we do in our main theorem. Nevertheless, our entropy satisfies similar properties, as will be shown below in Proposition \ref{lem:properties}.
Omitting normalization is also natural in other contexts---for example, for unbalanced optimal entropy-transport \cite{LieMieSav18} and in machine learning frameworks, where computing the normalization constants can be prohibitively expensive or formally impossible \cite{GoodBenCou16, JiaQin23}. 

\begin{prop}
\label{lem:properties}
    The $S^*$-entropy (\ref{eq:entropy}) satisfies the following properties:
    \begin{enumerate}
        \item (Nonnegativity) $S^*(\alpha, \mathcal{A}) \geq 0$.
        \item (Maximum entropy)  If $\mu_{\alpha}(\mathcal{A})\neq 0$ and $\mu_{\alpha}(\mathcal{A})< \vol(\mathcal{A})$, then 
        $$
        S^*(\alpha, \mathcal{A}) \leq \mu_{ \alpha}(  \mathcal{A})\log \left( \dfrac{C\, \vol ( \mathcal{A})}{\mu_{ \alpha}( \mathcal{A})}\right),
        $$
        where equality is achieved if and only if $\rho_{\alpha} = \mu_{ \alpha}(A)/\vol(\mathcal{A})$.
        \item (Set Monotonicity) For $\mathcal{A}_1 \subseteq \mathcal{A}_2, \;\;\; S^*(\alpha,\mathcal{A}_1) \leq S^*(\alpha, \mathcal{A}_2)$.
    \end{enumerate}
\end{prop}
\begin{proof}
(1) follows from the fact that $\rho_\alpha/C \in [0,1]$. For (2), we use Jensen's inequality and the concavity of the logarithm, which give, 
\begin{align*}
\dfrac{1}{\mu_\alpha(\mathcal{A})} S^*(\alpha, \mathcal{A}) +  \log \mu_\alpha (\mathcal{A}) = \dfrac{1}{\mu_\alpha(\mathcal{A})}\int_{\mathcal{A}} \rho_{\alpha} \log \left(\dfrac{\mu_\alpha(\mathcal{A})}{\rho_{\alpha} }\right) \; d\vol +\log C \leq \log \vol(\mathcal{A}) + \log C .
\end{align*}
Solving for $S^*(\alpha, \mathcal{A})$ gives the inequality.
Strict concavity implies a unique maximizer, the uniform distribution, $u = \rho(\mathcal{A})/\vol(\mathcal{A})$, which can be confirmed by computing its entropy.

For (3), we write $C_i := C_{\mathcal{A}_i}$ for $i=1,2$. Then, noting that $C_1 \leq C_2$, we have,
\begin{align*}
\label{eq:S-diff}
    S^*(\alpha, \mathcal{A}_2) - S^*(\alpha, \mathcal{A}_1) =& -\int_{\mathcal{A}_2} \rho_\alpha \log \left( \dfrac{\rho_\alpha }{C_2}\right) \; d\vol 
    + \int_{\mathcal{A}_1} \rho_\alpha \log \left( \dfrac{\rho_\alpha }{C_2} \cdot \dfrac{C_2}{C_1} \right) \; d\vol \\
    =& -\int_{\mathcal{A}_2\setminus \mathcal{A}_1} \rho_\alpha \log \left( \dfrac{\rho_\alpha }{C_2}\right) \; d\vol 
    + \mu_\alpha (\mathcal{A}_1) \log \left(\dfrac{C_2}{C_1} \right) \geq 0\, ,
\end{align*} 
where the first term is positive by (1), which completes the proof.
\end{proof}

\begin{definition}
Let $\mathcal{A}$ be an $n$-manifold embedded in $\mathbb{R}^{n+1}$.  Then $\operatorname{int}(\mathcal{A}) := \displaystyle \bigcup_i \operatorname{int_{JB}}(\mathcal{A}_i)$, where $\operatorname{int_{JB}}(\cdot)$ denotes the Jordan-Brouwer interior of a set \cite{GuiPol74} and $\{ \mathcal{A}_i \}$ is the set of  connected components of $\mathcal{A}$. 
\end{definition}

We can now define
\begin{definition}
\label{def:entropy}
Let $\mathcal{A}$ be an $n$-manifold embedded in $\mathbb{R}^{n+1}$ and $\alpha$ an $n$-form on $\mathbb{R}^{n+1}$. Then the {\it entropy} of $(\mathcal{A}, \alpha)$  is defined by,
    \begin{equation}
    \label{eq:entropy}
    S(\alpha, \mathcal{A}) = \begin{cases}
        S^*(\dd \alpha, {\rm int}(\mathcal{A})), \;\;\;\; {\rm if \;\,} \dd \alpha \neq 0 \\
        0, \;\,\;\;\;\;\;\;\; \;\;\;\;\;\;\;\;\; \;\;\;\;\;\;\; {\rm else}
    \end{cases} 
  .
    \end{equation}
\end{definition}
The entropy \eqref{eq:entropy} clearly satisfies all the properties in Proposition \ref{lem:properties}. Moreover, it satisfies the following

\begin{lem}[Stokes Compatibility]
\label{lem:stokes-compat}
    If $\alpha_1 \neq \alpha_2$ but $\dd \alpha_1 = \dd \alpha_2$ then $S(\alpha_1, \mathcal{A}) = S(\alpha_2, \mathcal{A})$, where  the forms and $\mathcal{A}$ are defined as in Definition \ref{def:entropy}.
\end{lem}
\begin{proof}
    The result follows immediately from the definition of entropy \eqref{eq:entropy}.
\end{proof}

\begin{remark}
The Stokes Compatibility property in Lemma \ref{lem:stokes-compat} affirms that our definition of entropy respects Stokes' theorem, which only recognizes  the information content of a form that is preserved by its exterior derivative, that is, modulo addition of a closed form.  
(Such constructions that depend only on the differential of the inputted function or form are also found, for example, in the case of the Dirichlet energy or the energy of harmonic maps \cite{Cor95}.)
Our entropy \ref{def:entropy} is therefore blind to topological information. Nevertheless, one can also account for topological data by, for example, defining the entropy,
$$
\mathcal{S}(\alpha, \mathcal{A}) := S^*(\dd \alpha, {\rm int}(\mathcal{A})) +S^*(\alpha_h, \mathcal{A}) \, ,
$$
where $\alpha_h$ is the harmonic component in the Hodge decomposition of $\alpha$. Though a study of such entropies is beyond the scope of this work, we do note that entropies of this sort, which sum the contributions due to the defining constraints or features of the given problem, can serve as powerful tools that can lead to novel results. A celebrated example is Perelman's $\mathcal{W}$-entropy, which sums the entropic contribution due to the Ricci flow for the Riemannian metric with that due to the `backward', or `conjugate', heat equation that the input function solves on the given manifold \cite{Per02, KleLot08}.


\end{remark}
\vspace{0.3cm}

Our main result shows that, among all Stokes-admissible submanifolds of codimension 1, $B \in \mathcal{B}_{\omega, \mathcal{M}}$, the given manifold's boundary maximizes the entropy.
\begin{thm}
\label{thm:entropy}
Let $\mathcal{M}\subseteq\mathbb{R}^n$ and let $\omega\in\Omega^{n-1}(\mathcal{M})$.  
For every Stokes–admissible set $B\in\mathcal{B}_{\omega, \mathcal{M}}$, we have
    \begin{equation}
    \label{eq:S-ineq}
         S(\omega,B) \,\leq \, S(\omega,\partial\mathcal{M})\, .
    \end{equation}
    Moreover, if ${\rm supp}(\rho_\omega) = \mathcal{M}$, then the maximum is  attained if and only if $B = \partial \mathcal{M}$ a.e.
\end{thm}
\begin{proof}
Without loss of generality, we assume $\mathcal{M}$ is such that $S(\omega,\partial\mathcal{M})<\infty$,  since otherwise the inequality is vacuously true. Similarly, if $\dd \omega = 0$, the result is likewise vacuously true with $S(\omega,B) = S(\omega,\partial\mathcal{M})=0$.
For $\dd \omega \neq 0$, by Proposition \ref{lem:properties}, since ${\rm int}(B) \subseteq \mathcal{M}$, $S(\omega, B) = S^*(\dd \omega,{\rm int}(B)) \leq S^*(\dd \omega, \mathcal{M})=S(\omega, \partial \mathcal{M})$.
Uniqueness follows from Proposition \ref{lem:properties} unless $\mu_{\dd \omega} (\mathcal{M} \setminus {\rm int}(B'))=0$ for some $B' \in \mathcal{B}_{\omega, \mathcal{M}}$, in which case $S(\omega, B') = S(\omega, \partial\mathcal{M})$. The condition in the theorem precludes this.
\end{proof}

\begin{remark}
\label{rem:stats-interpret}
    For a statistical interpretation of Theorem \ref{thm:entropy} one may consider $\mathcal{M}$ as the phase space, and $\dd \omega$ the weight function giving the relative (unnormalized) likelihood corresponding to every possible state. The candidate boundaries $B \in \mathcal{B}_{\omega, \mathcal{M}}$ are codimension-1 representations of $(\mathcal{M}, \dd \omega)$ obtained by integrating the information given by $\dd \omega$ over one of the dimensions. Theorem \ref{thm:entropy} is then the statement that any representation other than $\partial \mathcal{M}$ introduces new information, thereby reducing the volume of available phase space. 
    In physics, such an interpretation may also facilitate novel perspectives on, {\it e.g.}, black hole entropy \cite{Bek73,Bek74} and the holographic principle \cite{tHooft93, Bous02}, where the information in the interior volume of a space is viewed as being  encoded in its bounding surface. 
\end{remark}

We conclude with two examples.

\begin{example}[A piecewise constant function]
Given positive constants, $a,b,c,r$, let $X = [-2a,2a] \times [-b,b] \subset \mathbb{R}^2$, with 
$$
f(x,y) =  \begin{cases}
c, \,\,\,\,\,\,\,\,\,\,\,\,\,\,\,\, (x,y) \in [-a,a] \times [-b,b] \\
c/r, \,\,\,\,\,\,\,\,\,\,\,  (x,y) \in [a,2a] \times [-b,b]
\\
-c/r, \,\,\,\,\,\, (x,y) \in [-2a,-a] \times [-b,b]
\end{cases} .
$$
Let $Y = [-a,a] \times [-b,b]$ and $c/r\leq1$, so $C_X=1$. Then,
$$
S(f,X) = \vol(Y) c\log (1/c) + \vol(X\setminus Y)\, \left(\dfrac{c}{r} \right)  \log\left(\dfrac{r}{c}\right),
$$
so $S(f, X) \geq S(f, Y) = \vol(X\setminus Y)\, \left( \dfrac{c}{r} \right) \log\left(\dfrac{r}{c}\right)\geq0$.
One can check that this holds for all other  combinations of values of $c$ and $r$ ({\it e.g.}, $c>1$ and $r>1$, $c>1$ and $r<1$, {\it etc.}).

\end{example}

\begin{example}[Annulus and a Stokes–admissible candidate boundary]
\label{ex:annulus}
Let  $M=\{x\in\mathbb{R}^2 : r_i \leq |x| \leq r_o\}$
be an annulus, and let $\omega = f(r)\, d\theta$ with 
$f\in C^1([r_i,r_o])$.  Then
$d\omega = f'(r)\, dr\wedge d\theta$.
Using polar area $d\mathrm{vol}=r\,dr\,d\theta$, we have $\rho_{d\omega}(r)=\frac{|f'(r)|}{r}$.
The $S^*$-entropy of $d\omega$ on $M$ is therefore
$$
S^*(d\omega,M)
=
-\int_M 
\,\rho_{d\omega}\log\!\Bigl(\tfrac{\rho_{d\omega}}{C(M)}\Bigr)\,
d\mathrm{vol}
=
-2\pi\!\int_{r_i}^{r_o}
\frac{|f'(r)|}{r}\,
\log\!\Bigl(\tfrac{|f'(r)|/r}{C(M)}\Bigr)\, r\,dr,
$$
with 
$$
C(M)=
\max\left\{1,\operatorname*{sup}_{r_i\le r\le r_o}
\frac{|f'(r)|}{r}\right\}.
$$

Now let $B=\{ |x|=r_B\}$ with $r_i<r_B<r_o$ be a 
Stokes–admissible candidate boundary, so that
$\int_B \omega = \int_{\partial M}\omega = \int_M d\omega$.
Since 
\[
\int_B \omega = 2\pi f(r_B), 
\qquad 
\int_{\partial M}\omega 
= 2\pi\bigl(f(r_o)-f(r_i)\bigr),
\]
we obtain the constraint,
\[
f(r_B)=f(r_o)-f(r_i),
\]
which (for monotone $f$) determines a unique $r_B\in(r_i,r_o)$.

The interior associated to $B$ is the smaller annulus, ${\rm int}(B)=\{r_i\le r \le r_B\}$.
Without loss of generality we assume that $d\omega\neq0$, so the entropy is given by $S(\omega,B)=S^*(d\omega,\operatorname{int}(B))$.
Therefore,
$$
S(\omega,B)
=
-2\pi\!\int_{r_i}^{r_B}
\frac{|f'(r)|}{r}\,
\log\Bigl(\tfrac{|f'(r)|/r}{C(\operatorname{int}(B))}\Bigr)\, r\,dr,
$$
where
$$
C(\operatorname{int}(B))
=
\max\left\{1,\operatorname*{sup}_{r_i\le r\le r_B}
\frac{|f'(r)|}{r}\right\}.
$$

The integrand is the same in both expressions, $\operatorname{int}(B)\subset M$ reduces the radial interval of integration, and $C(M)\geq C(\operatorname{int}(B))$ increases the logarithmic factor, so we have
$S(\omega,B) \leq S(\omega,\partial M)$
for any function, $f(r)$. 
\end{example}
\vspace{0.5cm}

The entropy-extremizing perspective for Stokes' theorem presented here suggests a broader framework of {\it entropy-extremizing dualities}. In this note we have focused on codimension-$1$ submanifolds 
$B \subset \mathcal{M}$ satisfying the Stokes relation, and showed that the canonical boundary 
$\partial \mathcal{M}$ maximizes entropy.
Instead of the space of submanifolds satisfying the Stokes relation, we can similarly consider the space of operators that globally satisfy the Stokes relation: Given $\omega \in \Omega^{n-1}(\mathcal{M})$, one may define the alternative 
derivative operators $\mathfrak{d}:~\Omega^{n-1}(\mathcal{M}) \to \Omega^n(\mathcal{M})$ such that
$$
   \int_{\mathcal{M}} \mathfrak{d} \omega \;=\; \int_{\partial \mathcal{M}} \omega \;,
$$
that is, operators indistinguishable from the exterior derivative, $\dd$, when applied to $\omega$ and integrated over $\mathcal{M}$. 
Preliminary investigations indicate that among all such $\mathfrak{d}$, entropy is similarly uniquely maximized for $\mathfrak{d} = \dd$. 

These observations point toward a more comprehensive theory of entropy-extremizing dualities that could 
encompass classical dualities from convex analysis, Hodge theory, and functional inequalities, and that may, more importantly, provide a rigorous variational framework for the principled discovery of new results. A systematic development of this framework is the focus of future work.
\\
\\

\subsection*{Acknowledgments.} 
The author would like to thank Henry Cohn for his feedback and mentorship, and William Minicozzi and David Ebin for their comments and feedback on earlier drafts.
\\

 \bibliography{refs}
 \bibliographystyle{alpha}
 \end{document}